\documentclass{article}
\usepackage{amsfonts,amssymb}
\newtheorem{theorem}{Theorem}
\newtheorem{lemma}[theorem]{Lemma}
\newtheorem{example}[theorem]{Example}

\newcommand{\Tr}{{\rm Tr}}
\newcommand{\erf}{{\rm erf}}
\newcommand{\Ei}{{\rm Ei}}
\def\text#1{\hbox{\rm#1}}

\newenvironment{proof}{{\it Proof}:}{\hfill$\square$\\ $\phantom{A}$\\}

\begin{document}

\title{Integration in terms of exponential integrals and incomplete
gamma functions I.
}
\author{
Waldemar Hebisch \\
       {Mathematical Institute}\\
       {Wroc\l{}aw University}\\
       {50-384 Wroc\l{}aw}\\
       {\tt hebisch@math.uni.wroc.pl}
}
\date{}
\maketitle
\begin{abstract}
This paper provides a Liouville principle for integration in
terms of exponential integrals and incomplete gamma functions.
\end{abstract}




\section{Introduction}

Indefinite integration means that given $f$ in
some set we want to find $g$ from possibly larger set such that
$f = g'$.  The first step is to delimit possible form
of $g$.  In case when both $f$ and $g$ are elementary
the Liouville-Ostrowski theorem says that only new transcendentals
that can appear in $g$ are logarithms.  More precisely,
when $f\in L$ where $L$ is a differential field with
algebraically closed constant field and $f$ has integral
elementary over $L$, then
$$
f = v_0' + \sum c_i\frac{v_i'}{v_i}
$$
where $v_i \in L$ and $c_i\in L$ are constants.

Liouville-Ostrowski theorem was a starting step
for elementary integration (\cite{Risch:Prob},
\cite{Risch:Sol}, \cite{Daven}, \cite{Tra:alg},
\cite{Bro:ele}).
However, it is interesting
to take larger class of integrands $g$ allowing
special functions in the answer.  Attempts to do this
started quite early.  Starting point for us were
\cite{SSC} which forms basis for early work on
integration in terms of logarithmic integrals and
error functions (\cite{Che:erf}, \cite{Kno:liou1} and \cite{Kno:liou2}).
We extend previous results allowing incomplete gamma function
$$\Gamma(a, x) = \int_x^\infty t^{a-1}\exp(-t)dt.$$
Since
$$\erf(x) = 1 - \pi^{-1/2}\Gamma(1/2, x^2)$$
incomplete gamma function is more general than error function.
Let us add that Liouville principle in \cite{SSC} from one
point of view is very general and handles large class of special
functions, however this class has small intersection with
classical special functions.  So for the purpose of integration
in terms of classical special functions \cite{SSC} has
limited use.

In first part we develop structural theory of exponentials
and logarithms in a differential field which is of
independent interest.  We give conditions
which propagate well trough towers of extensions.  As
result we generalise Risch structure theorem to large
class of differential fields.  Based on this we prove
Liouville type principle for integration in terms
of exponential integral and gamma functions.  This
principle leads to necessary conditions which are
satisfied by parts of integral.  In fact, development
of ideas from this paper lead to a complete integration algorithm
for integration in terms of exponential integral
and incomplete gamma functions.  We announced first version
of the integration algorithm at ISSAC 2015 \cite{He:CCA}
and we will present details of the algorithm in a separate paper.
We hope that other classes
of special functions can be handled in a similar way.
Polylogarithms are of particular interest, however they raise
tricky theoretical questions.

Our main interest is allowing larger class of integrals.  However
structural theory should be also helpful in handling larger
classes of integrands.  We should mention here \cite{Raab}
where elementary integration is extended to larger class
of differential fields (in particular allowing non-Liouvillian
special functions).

After writing first version of this paper we learned about closely
related paper \cite{LeLa} by U. Leerawat and V. Laohakosol.  They
define class of extensions called Ei-Gamma extensions which
generalised extensions from \cite{SSC} by allowing directly
exponential integrals and incomplete gamma functions with rational
first argument.  Our work also allows incomplete gamma functions
with irrational first argument so is more general in this aspect.
If we ignore incomplete gamma functions with irrational first argument,
then Theorem 3.1 in \cite{LeLa} would be more general than our
Theorem 13.  Our paper deals only with exponential
integrals and incomplete gamma functions which means that
we can use trace in straightforward way, avoiding extra
arguments in \cite{SSC} and \cite{LeLa} needed to handle
nonclassical special functions.  As our work is intended as first
step towards integration algorithm we need more structural information
about integrals which we give in Theorems 18, 19 and 21.

\section{Preliminaries}

We assume standard machinery of differential fields
(see for example \cite{Ros}).

If $K$ is a differential field, $v, u \in K$ and
$v' = u'/u$ we say that $u$ is exponential
element and that $v$ is logarithmic
element.  The set of exponential elements in $K$
is an abelian group with respect to
multiplication.  The set of logarithmic elements in $K$
is an abelian group with respect to addition.

We say that a differential field
$k$ is l-closed iff for every algebraic extension
$E$ of $k$ if $c_i\in E$ are constants linearly
independent over rational numbers, $u_i\in E$,
$\sum c_iu_i'/u_i = v'$, then some powers
of $u_i$ are in $k$ modulo multiplicative constants.
Natural constructions of differential
fields lead to l-closed fields, in particular
algebraic extension, extension by logarithm,
extension by exponential and extension by
nonelementary primitive preserve l-closed fields
(for algebraic extension this is immediate consequence of
definition, the other are proved in the sequel).
However, it is easy
to build artificial examples which are not
l-closed, so we state this assumption explicitly.

We say that a differential field $k$ is log-explicit
if equation $\sum c_i {u_i' \over u_i} = v'$ with
$u_i$ and $v$ in an algebraic extension $E$ of $k$ and $c_i \in E$
being constants
linearly independent over rational numbers implies
that $u_i$ are exponential elements in $E$.

Let $K$ be a field.  We say that $f\in K(\theta)$ is a proper
rational function iff degree of numerator of $f$ is smaller
than degree of denominator.

The following lemma is well-known, so we give it without proof.

\begin{lemma}\label{decomp1}  Let $K$ be a field and $f \in K(\theta)$.
We can write $f = w + p$ where $w \in K[\theta, \theta^{-1}]$
and $p$ is a proper rational function with denominator
relatively prime to $\theta$.  Such decomposition
is unique.
\end{lemma}

The method of proof used in the next lemma is well-known, but we give
proof for readers convenience.

\begin{lemma}\label{decomp2}  Let $K$ be a differential field, $\theta$
an exponential over $K$, that is $\theta' = \eta'\theta$
with $\eta \in K$ and $\theta$ is transcendental over $K$.
Assume that
$$
f = \sum c_i\frac{v_i'}{v_i}
$$
with $v_i \in K(\theta)$ and constant $c_i$.  Then
there exist constant $a$ and ${\bar v}_i \in K$ such that
$$
f = a\eta' + \sum c_i\frac{{\bar v}'}{\bar v} + p
$$
where $p$ is a proper rational function with denominator
relatively prime to $\theta$.
\end{lemma}
\begin{proof}  Without loss of generality we may assume that
all $v_j$ are polynomials (just consider separately contribution
from numerator and denominator).  Write $v_i = {\bar v}_is_i$
where $s_i$ is monic.  We have
$$
\frac{s_i'}{s_i} = l_i\frac{\theta'}{\theta} + \text{proper rational function}
$$
where $l_i$ is degree of $s_i$ and proper rational part has
denominator relatively prime to $\theta$.  Next
$$
\frac{v_i'}{v_i} = \frac{{\bar v}'}{\bar v} + \frac{s_i'}{s_i}
= \frac{{\bar v}'}{\bar v} + l_i\frac{\theta'}{\theta}
 + \text{proper rational function}
$$
so
$$
\sum c_i\frac{v_i'}{v_i} = \sum c_i\frac{{\bar v}'}{\bar v} +
\sum c_il_i\eta' + \text{proper rational function}
$$
and we get the result with $a = \sum c_il_i$.
\end{proof}

In similar way as Lemma \ref{decomp2} we can prove the following:

\begin{lemma}\label{decomp3}  Let $K$ be a differential field, $\theta$
primitive over $K$, that is $\theta' \in K$ and $\theta$ is
transcendental over $K$.
Assume that
$$
f = \sum c_i\frac{v_i'}{v_i}
$$
with $v_i \in K(\theta)$ and constant $c_i$.  Then
there exist ${\bar v}_i \in K$ such that
$$
f = \sum c_i\frac{{\bar v}'}{\bar v} + p
$$
where $p$ is a proper rational function.
\end{lemma}

\section{Structure of fields}

First we show that property of being l-closed
or log-explicit
is preserved by common extensions.  Since the
proofs are quite similar we will simultaneously
state two versions, omitting text about log-explicit in parentheses
one will get statements about l-closed fields.
Replacing l-closed by log-explicit one will get second
statement.

\begin{lemma}\label{l-closed1}
Let $k$ be l-closed (or log-explicit) differential field of
characteristic $0$, $L$ is an extension of $k$
with the same constants.
Assume that
$$\sum_{i=1}^n c_i {u_i'/u_i} = v'$$
with $u_i, v$ algebraic over $L$, and $c_i$ being constants
algebraic over $k$ and
linearly independent over rational numbers implies
that $u_i$ and $v$ are algebraic over $k$.  Then
$L$ is l-closed (respectively log-explicit).
In particular the condition above is satisfied when 
$t$ is a nonelementary primitive over $k$, and $L = k(t)$.
\end{lemma}
\begin{proof}
The first part is clear.  To prove claim about $k(t)$
let $E$ be algebraic over $k(t)$.  Since $t$ is nonelementary
any new constant in $E$ is algebraic over $k$.  Without loss
of generality we may assume that all constants in $E$ are
in $k$ (we enlarge $k$ if needed).
  Assume
$\sum_{i=1}^n c_i {u_i'/u_i} = v'$ with $u_i, v \in E$, 
and $c_i$ being constants
linearly independent over rational numbers.   By \cite{Ros} Theorem 2
the $u_i$ are algebraic over $k$ and there is a constant $c$
such that $f = v -ct$ is algebraic over $k$.  If $c = 0$, then
our assumption about $L=k(t)$ holds.
If $c \ne 0$, then
$t = (1/c)(v - f)$ and  $t' = (1/c)(-f' + \sum c_i {u_i' \over u_i})$,
so $t$ is elementary, which contradicts assumption that $t$
is nonelementary.
\end{proof}

\begin{lemma}\label{l-closed2}
If $k$ is l-closed (or log-explicit) differential field of
characteristic $0$, $K$ is algebraic over $k(s, t)$,
$t'/t = s'$, transcendental degree of $K$ over $k$ is $1$,
$K$ has the same constants as $k$,
then $K$ is l-closed (respectively log-explicit).  
\end{lemma}
\begin{proof}
Let $E$ be algebraic over $K$.
Any new constant in $E$ is algebraic over $k$.  Without loss
of generality we may assume that all constants in $E$ are
in $k$ (we enlarge $k$ if needed).
  Assume
$$\sum_{i=1}^n c_i {u_i'/u_i} = v'$$ with $u_i, v \in E$ and 
$c_i$ being constants 
linearly independent over rational numbers.  By \cite{Ros} Theorem 1
forms $\omega_1 = {dt \over t} -ds$ and
$dv - \sum c_i{du_i \over u_i}$ are linearly dependent over
constants.  Since at least one of $t$ and $s$ is transcendental over
$k$, we have $\omega_1 \ne 0$, so there is a constant $c$ such
that $$d(v - cs) + c{dt \over t} - \sum c_i{du_i \over u_i} = 0.$$
If $c$ and $c_i$ are linearly
independent over rationals,
then by \cite{Ros} Proposition 4 this implies that $v - cs$,
$t$, and $u_i$ are algebraic over $k$.  So
$$-ct'/t + \sum_{i=1}^n c_i {u_i'/u_i} = (v - cs)'$$
is equality with terms algebraic over $k$ and we can
use our assumption about $k$.
If $k$ is l-closed some power of
$u_i$ is in $k$ modulo multiplicative constants, which
is what we require.  If $k$ is log-explicit than $u_i$
are exponential elements in $E$, which again is what we need.
In case when $c$ are $c_i$ are linearly
dependent over rationals, $c$ is a linear combination
over rationals of $c_i$ so we can write
$mc = \sum r_i c_i$ with $m$, $r_i$ being integers.  Then
$$md\/(v - cs) - \sum_{i=1}^{n} c_i{dw_i \over w_i} = 0,$$
where
$w_i = u_i^m/t^{r_i}$.  By \cite{Ros} Proposition 4 this
means that $w_i$ and $v - cs$ are algebraic over $k$.  Again, if
$k$ is l-closed some powers of $w_i$ are in $k$ modulo
multiplicative constants.  But $u_i^m = w_it^{r_i}$, so some
powers of $u_i$ are in $k(t)$ modulo multiplicative constants.
If $k$ is log-explicit, then $w_i$ are exponential elements in $E$.
But $t$ is also exponential element in $E$ so $u_i^m$ are
exponential elements in $E$.  Hence $u_i$ are
exponential elements in $E$.
\end{proof}

In particular Lemma \ref{l-closed2} handles extension by
an exponential, by a logarithm or by Lambert $W$ function.

\begin{lemma}\label{log-closed}
If $k$ is l-closed differential field of
characteristic $0$, $K$ is an
algebraic extension of $k$ and $v$ is a logarithmic
element in $K$, then there is a constant $c$ such
that $v - c \in k$.
\end{lemma}
\begin{proof}
By definition, if $v$ is logarithmic element in $K$, then there
exists $u \in K$ such that $u'/u = v'$.  Since $k$
is l-closed there exist integer $n$ such that
$u^n = c_0w$ for some $w \in k$ and a constant $c_0$.  Then
$$
u'/u = (1/n)w'/w \in k
$$
so $v' \in k$.  Taking trace from $K$ to $k$ we get
$$
mv' = \Tr(v)'
$$
when $m$ is degree of $K$ over $k$.  Consequently putting
$f = \Tr(v)/m$ we have $f\in k$ and $f' = v'$ so
$v - f = c$ is a constant.
\end{proof}

\begin{lemma}\label{exp-const}
If $K$ is extension of a differential field $k$ by algebraic constants,
$u \in K - \{0\}$ is such that $u'/u \in k$, then
there is a constant $c\in K$ such that $u/c \in k$
\end{lemma}
\begin{proof}
Without loss of generality we may assume that $K$
is finitely generated over $k$.  By induction
we may assume that $K$ is generated over $k$
by a single algebraic constant $\alpha$, so
that $K = k(\alpha)$.  Then powers of $\alpha$
form a basis of $K$ over $k$ and we can write
$$
u = \sum_{i=0}^{n-1}\alpha^i w_i
$$
where $n$ is degree of $\alpha$ over $k$ and
$w_i \in k$.  Next, let $v = u'/u$.  We have
$u' - vu = 0$, so
$$
\sum_{i=0}^{n-1}\alpha^i(w_i' -vw_i) = 0.
$$
Consequently for each $i$ we have $w_i' -vw_i = 0$.
Since $u \ne 0$ there is $i$ such that $w = w_i \ne 0$.
Then for each $i$ we have $(w_i/w)' = 0$, so
$\beta_i = w_i/w$ are constants and we have
$$
u = \sum_{i=0}^{n-1}\alpha^i \beta_iw = w\sum_{i=0}^{n-1}\alpha^i \beta_i
= cw
$$
with $c = \sum_{i=0}^{n-1}\alpha^i \beta_i$ being a constant.
\end{proof}

We will need the following lemma, which is a specialised
version of \cite{Ros} Theorem 2:

\begin{lemma}\label{ros-struct}
Let $E$ be a differential field of characteristic $0$, $F$
a differential extension of $E$ having the same constants
with $F$ algebraic over $E(\eta)$ for some given $\eta \in F$.
Assume that $E$ algebraically closed in $F$.
Suppose that $u, v \in F$ are such that $\frac{u'}{u} = v'$.
If $\eta' \in k$, then $u \in E$ and there is
constant $c$ such that $v - c\eta \in E$.  If
$\frac{\eta'}{\eta} \in E$, then $v \in E$ and there
are integers $n$ and $j$ such that $\frac{u^n}{\eta^j}\in E$.
\end{lemma}

Now we give examples where our properties are violated.

\begin{example}  Let $C = {\mathbb Q}(a, b)$ where $a$ and $b$ are
transcendental constants.  Let $E = C(x, v, w)$ where $x' = 1$,
$v' = 1/(x - a)$, $w' = 1/(x - b)$.  That is, $C$ is field of
constants and we extend $C(x)$ by $\log(x - a)$ and $\log(x - b)$.
Let $\sigma$ be an automorphism of $E$ such that $\sigma(a) = b$,
$\sigma(b) = a$, $\sigma(x) = x$, $\sigma(v) = w$, $\sigma(w) = v$.
It is easy to check that $\sigma$ is a differential automorphism
and $\sigma^2$ is identity.  Let $F = \{f \in E: \sigma(f) = f\}$.
Then $E$ is algebraic of degree $2$ over $F$.  By Lemma \ref{l-closed2}
$E$ is log-explicit and l-closed.  By definition of log-explicit
also $F$ is log-explicit.  However, $F$ is not l-closed.
Namely, $v$ is a logarithmic element in an algebraic extension
of $F$.  If $F$ were l-closed by Lemma \ref{log-closed} there
would be constant $c$ such that $v + c \in F$.  But
$\sigma(v+C) = w+C \ne v+C$, so this would be a contradiction.
\end{example}

\begin{example}  This is variant of predator-prey example
from \cite{PrSi}.
  Let $C = {\mathbb Q}(a, b)$ where $a$ and $b$ are
transcendental constants.  Let $E = C(x, y)$.
We define derivation on $K$ by
formulas
$$
x' = 1,
$$
$$
y' = \frac{-by + xy}{ax - xy}.
$$
By routine calculation we check that
$$
(x + y)' - b\frac{x'}{x} - a\frac{y'}{y} = 0.$$
It is easy to check that differential forms $\omega_1 = dx$ and
$$
\omega_2 = d(x + y) - b\frac{dx}{x} - a\frac{dy}{y}
$$
are linearly independent over $K$, so by \cite{Ros} Theorem 1
transcendental degree of $K$ over constants is at least $2$.
This means that constant field of $K$ is $C$.  Now, suppose
$$
v_0' + \sum c_i\frac{v_i'}{v_i} = 0
$$
with $v_i$ in an algebraic extension $L$ of $K$.
By \cite{Ros} Theorem 1 there are constants $\alpha$ and $\beta$
such that
$$
dv_0 + \sum c_i\frac{dv_i}{v_i} - \alpha\omega_1 - \beta \omega_2 = 0
$$
in $\Omega_{L/C}$.  If $\beta = 0$ this means that $v_0 - \alpha x$ and $v_i$
for $i>0$ are constants.  Otherwise without loss of generality
we may multiply $v_0$ and $c_i$ by $1/\beta$ and in the following
only consider $\beta = 1$.

Consider linear space $V$ over ${\mathbb Q}$ spanned
by $-b$, $-a$, and $c_i$.  Without loss of generality we may
assume that $e_1 = -b$, $e_2 = -a$, and $e_i = c_i$ for $i>2$
form basis of $V$ (we renumber $c_i$-s and prepend zeros if needed).  
Now we write $c_1$ and $c_2$ as rational linear combination of $e_i$-s.
Let $m$ be common denominator of coefficients of linear combination.
Then we can write $mc_i = \sum n_{i,j}e_j$ with integer $n_{i,j}$.
Put $w_1 = x$, $w_2 = y$, $w_i = v_i$ for $i > 2$.
We get equation
$$
d(v_0 - (x + y) - \alpha x) + \sum_i e_i
\left(\frac{dw_i}{w_i} + \frac{1}{m}\sum_{j=1}^2n_{j,i}\frac{dv_j}{v_j}
\right)
$$
which implies
$$
md(v_0 - (x + y) - \alpha x) + \sum_i e_i\frac{du_i}{u_i} = 0
$$
where $u_i = w_i^m\prod_{j=1}^2v_j^{n_{j,i}}$.
By \cite{Ros} Proposition 4 this implies that $u_i$
are algebraic over $C$, hence are constant.  Also, note
that matrix $\{n_{i,j}\}$ with $i=1,2$, $j=1,2$ must be
invertible over ${\mathbb Q}$ (otherwise $dx/x$ and $dy/y$
would be linearly dependent over constants).  So modulo
multiplicative constants some powers of $v_i$, $i=1,2$
are power products of $x$ and $y$.  In other words
modulo multiplicative constants some powers of $v_1$
and $v_2$ are in $K$.  Using expression for $u_i$ with
$i > 2$ and known fact about $v_1$ and $v_2$ in similar
way we see that some powers of $v_i$ for $i>2$ are in $K$
modulo multiplicative constants.  So $K$ is l-closed.
In the same way we show that there are no nonconstant
exponential elements in $L$.  Namely, equality
$v_0' = v_1'/v_1$ means $\beta = 0$ (otherwise matrix
$\{n_{i,j}\}$ above would be singular which leads to
contradiction).  Now \cite{Ros} Proposition 4 this
implies that $v_1$ is algebraic over constants, hence
constant.  So $K$ is not log-explicit.
\end{example}


In case of log-explicit fields we can strengthen Lemma \ref{ros-struct}:
\begin{lemma}\label{ros-struct2}
Assume that $k$ is log-explicit differential field of
characteristic $0$, $K$ is algebraic over
$k(s, t)$, $t'/t = s'$, transcendental degree of $K$ over $k$ is $1$,
$K$ has the same constants as $k$ and $k$ is algebraically closed in $K$.
Suppose that $u, v \in K$ are such that $\frac{u'}{u} = v'$.
Then there exist rational number $r$ such that
$v - rf \in k$.
\end{lemma}
\begin{proof}
As in proof of Lemma \ref{l-closed2} we consider differential
forms $\frac{dt}{t} - ds$ and $\frac{du}{u} - dv$ getting
equality
$$
\frac{du}{u} - cs\frac{dt}{t} - d(v - cs)
$$
with constant $c$.  By \cite{Ros} Proposition 4
we have $v - cs$ algebraic over $k$, hence in $k$.
If $c$ is rational, then we are done.
Otherwise, if $c$ is irrational by \cite{Ros} Proposition 4
we see that $v - cs$, $u$ and $t$ are algebraic over $k$,
hence in $k$.
We write
$$
\frac{u'}{u} - c\frac{t'}{t} - (v - cs)' = 0.
$$
Since $k$ is log-explicit
$t$ is an exponential element in $k$.  In other words,
there is $\eta\in k$ such that $\eta' = t'/t$.  Then
$(\eta - s)' = 0$.  Since $K$ have the same constants as $k$
we have $s \in k$. But this gives contradiction with assumption that
$K$ has transcendental degree $1$ over $k$.
\end{proof}

Our results lead to generalisation of Risch structure theorem:

\begin{lemma}\label{risch-str}
Assume that $k$ is l-closed differential field of characteristic $0$.
Let $W$ be a linear space over rational numbers ${\mathbb Q}$ generated
by derivatives of
logarithmic elements in $K$.  Assume that $k(\theta)$ is
extension of $k$ with the same constants.  If $f \in k$,
$\theta' = \frac{f'}{f}$ then $\theta$ is transcendental over $k$
if and only if $\frac{f'}{f} \notin W$.  Similarly,
if $f \in E$, $\theta' = f'\theta$, then $\theta$ is
transcendental over $k$ if and only if $f' \notin W$.
\end{lemma}

\begin{proof}
When $\theta' = \frac{f'}{f}$ and $\frac{f'}{f} \in W$, then
there exist logarithmic element $v \in k$ and integer $n$ such
that $n\frac{f'}{f} = v'$.  Then $(n\theta - v)' = 0$ so
$(n\theta - v)$ is a constant, hence $c = (n\theta - v) in k$,
so $\theta = (c + v)/n \in k$.  Similarly, when $\theta' = f'\theta$,
and $f' \in W$,
then there exist logarithmic element $v \in k$ with corresponding
exponential element $u$ and integer $n$ such that
$nf' = v$.  Then
$$
\left(\frac{\theta^n}{u}\right)' = 0
$$
so it is a constant, hence in $k$.  Consequently $\theta^n = cu$,
so $\theta$ is algebraic over $K$.  This proves implication in
one direction.

It remains to prove that when $\theta$ is algebraic over $k$, then
$\frac{f'}{f} \in W$ (respectively $f' \in W$).
When $\theta' = \frac{f'}{f}$
write $u = f$ and $v = \theta$.
When $\theta' = f'\theta$ write $u = \theta$ and $v = f$.
Then $v$ is logarithmic element in $k(\theta)$ and $u$ is
corresponding exponential element.  Since $k(\theta)$ is
algebraic extension of $k$ and $k$ is l-closed some power $u^m$ of
$u$ is in $k$ modulo multiplicative constants.  But $k(\theta)$
and $k$ have the same constants, so in fact $u^m \in k$.  Then
$mv$ is a logarithmic element in $k$ so $v \in W$.
\end{proof}

{\bf Remark}.  Of course utility of Lemma \ref{risch-str} depends
on our ability to compute space $W$.  For towers it is natural
to use inductive procedure.  If in given step we get no new
exponential elements (for example when
we add a nonelementary primitive), then the space $W$ remains
unchanged.  Similarly, in case of algebraic extension space
$W$ remains unchanged.  Lemma \ref{ros-struct2} covers case of extensions
by exponential, by logarithm and by Lambert $W$ function.

\section{Structure of integrals}

We say that a differential field $E$ is a gamma extension of
$F$ if there exists $\theta_1, \dots, \theta_n \in E$ such that
$E = F(\theta_1, \dots, \theta_n)$ and for each $i$, $1\leq i \leq n$,
one of the following holds
\begin{enumerate}
\item $\theta_i$ is algebraic over $F_i$
\item $\frac{\theta_i'}{\theta_i} = u'$ for some $u \in F_i$
\item $\theta_i' = \frac{u'}{u}$ for some $u \in F_i$
\item $\theta_i' = \frac{u'}{v}$ for some $u, v \in F_i$ such that
$u' = v'u$
\item there are $w, u, v \in F_i$ and integers $k, m$ such that
$\theta_i' = wu'$, $u' = v'u$, $w^k = v^m$, $-k< m < 0$.
\item there are $w, u, v \in F_i$ and an irrational constant $a$
 such that $\theta_i' = w'u$, $u' = ((a - 1)v' - w')u$,
 $v' = w'/w$.
\end{enumerate}
where $F_i = F(\theta_1, \dots, \theta_{i-i})$.  Intuitively
clauses 2 to 6 above mean $\theta_i = \exp(u)$,
$\theta_i = \log(u)$, $\theta_i = \Ei(v)$,
$\theta_i = \Gamma((k + m)/k, -v)$ or 
$\theta_i = \Gamma(a, v)$ respectively.
However, definitions above are a bit more general.  Namely,
let $u = x$ and $\theta = c\exp(x)$ where $c \ne 0$ is a constant.
We have $\frac{\theta'}{\theta} = u'$, so condition 2 above
is satisfied.  In other words, differential conditions
define exponential up to multiplication by constant.
Similarly, logarithm is defined only modulo additive
constant.

If the point 6 above is satisfied with $F_i$ replaced by $F$
abusing notation we
say that $w'u$ is an irrational gamma term over $F$
(strictly speaking we should consider tuple $(a, w, v, u)$).

If in the definition of gamma extension we allow only cases 1 to 5,
then we call the resulting extension rational gamma extension.

\begin{theorem}\label{gamma-str0}
Let $E$ be a log-explicit differential field of
characteristic $0$.
If $f$ has
integral in a gamma extension of $E$, then
there is algebraic extension $\bar E$ of $E$ such that
we can write
\begin{equation}\label{inteq}
f = v_0' + \sum_{i\in I_1} c_i\frac{v_i'}{v_i} +
    \sum_{i \in I_2} c_i v_i'\frac{u_i}{v_i} +
    \sum_{i \in I_3} c_iv_i'w_iu_i +
    \sum_{i \in I_4} c_iw'_iu_i
\end{equation}
where $c_i\in  \bar E$ are constants, 
$v_0, v_i, u_i, w_i\in \bar E$, $u_i' = v_i'u_i$, for
$i \in I_2\cup I_3$, $u_i' = ((a_i - 1)v_i' - w_i')u_i$
for $i \in I_4$, $a_i \in \bar E$ are irrational constants,
$v_i' = w_i'/w_i$ for $i \in I_4$,
 $k_i$
are positive integers,
$w_i^{k_i} = v_i^{m_i}$ for $i \in I_3$, with $m_i$ being
integers such that $-k_i< m_i < 0$ and $k_i$ and $m_i$ are relatively
prime.

If $f$ has integral in a rational gamma extension, then the
sum over $I_4$ does not appear in the result.  Also in
such case we
can
drop assumption that $E$ is log-explicit.
\end{theorem}

To prove Theorem \ref{gamma-str0} we use methods of Theorem 1.1
from \cite{SSC}.  By assumption equality \ref{inteq}
holds, but with $v_i$, $u_i$ and $w_i$ in some gamma extension
of $E$.  Gamma extension is a tower with each step either
adding algebraic element, or a logarithm or an exponential
or a nonelementary primitive.  Proceeding inductively it is
enough to prove that if $f \in E$ and $v_i, u_i, w_i$ are in
$F$ which is algebraic over $E(\eta)$, then we can find
$v_i, u_i, w_i$ in $E$.  Before actual proof we need several
lemmas.

\begin{lemma}\label{irr-gamma-exp}
Let $E$ be a differential field of characteristic $0$.
Assume $F$ is algebraic over $E(\eta)$, $\eta$ is an exponential
over $E$, $E$ is algebraically closed
in $F$ and $F$ has the same constants as $E$.
If $w'u$ is an irrational gamma term in $F$, then $w \in E$, $v \in E$ and
there are integers $n$ and $j$ such that
$\frac{u^n}{\eta^j} \in E$.
\end{lemma}

\begin{proof}
Applying Lemma \ref{ros-struct} to $u$ and $(a-1)v - w$ we get
$\frac{u^n}{\eta^j} \in E$ for appropriate $n$ and $j$
and $(a-1)v - w \in E$.  Similarly, applying Lemma \ref{ros-struct} to
$w$ and $v$ we get $v \in E$.  Consequently also $w \in E$.
\end{proof}

\begin{lemma}\label{gamma-str1-exp}
Let $E$ be a differential field of characteristic $0$, $F$ be a differential
field algebraic over $E(\eta)$, $\eta'/\eta = v'$ with $v\in E$.
Suppose that $F$
has the same constants as $E$ and equation \ref{inteq} holds with $f\in E$,
$v_i, u_i, w_i \in F$, then we can find new
$v_i, u_i, w_i$ algebraic over $E$ and new constant $c_i$
such that \ref{inteq} holds.
\end{lemma}
\begin{proof}
Replacing $E$ by its algebraic closure in $F$ we may assume
that $E$ is algebraically closed in $F$.
For $i \in I_2 \cup I_3$ we have $\frac{u_i'}{u_i} = v_i'$, so
by Lemma \ref{ros-struct} we have $v_i \in E$ and there are
integers $n_i, j_i$ such that $u_i^{n_i}/\eta^{j_i} \in E$.
For $i \in I_4$ we have $u_i^{n_i}/\eta^{j_i} \in E$ by
Lemma \ref{irr-gamma-exp}.
Without loss of generality we may assume that $n_i = 1$.  Namely,
let $n$ be the least common multiple of $n_i$.  We extend $F$ by
adding element $\bar \eta$ such that ${\bar \eta}^n = \eta$.
Then, ${\bar \eta}'/{\bar \eta} = \eta'/(n\eta) = v'/n$,
so if we replace $F$ by $F({\bar \eta})$ and $\eta$ by ${\bar \eta}$
we still have fields of required form.  After that
$u_i/{\bar \eta}^{j_in/n_i}$ is algebraic over $E$, so in $E$.
We shall keep old notation but from now on we will assume
that $n_i = 1$, so in particular $u_i \in E(\eta)$.

Note that for $i\in I_3$ corresponding $w_i$ are algebraic
over $E$ extended by $v_i$, so also $w_i$ are algebraic over
$E$, hence in $E$.  For $i \in I_4$ by Lemma \ref{irr-gamma-exp}
$w_i \in E$. Together with previous remarks it means
that terms of sum over $i \in I_2 \cup I_3 \cup I_4$ are all in $E(\eta)$.
Consequently, by taking trace of both sides of equation \ref{inteq}
we may assume that
$v_0$ and $v_i$ for $i\in I_1$ are in $E(\eta)$, so all terms
are in $E(\eta)$.  Now we use expansion into partial fractions.
For $i \in I_2\cup I_3 \cup I_4$ corresponding terms
are in $E[\eta,\eta^{-1}]$.
Terms of sum over $I_1$ are logarithmic derivatives, so have
only simple poles.  All finite poles of $v_0'$ are multiple, so except
possibly for pole at $\eta = 0$ can not cancel with other
terms.  Consequently $v_0$ has no finite poles except
possibly at $\eta = 0$, so $v_0 \in E[\eta,\eta^{-1}]$.
By Lemma \ref{decomp2} sum over $I_1$ is equal to sum with terms
in $E$ plus proper rational function with denominator
relatively prime to $\theta$.  By Lemma \ref{decomp1}
decomposition of $f$ into proper rational function with denominator 
relatively prime to $\theta$ and element of $E[\eta,\eta^{-1}]$ is unique,
so proper rational part must be $0$.  Now, we have equality
with all terms in $E[\eta,\eta^{-1}]$,
so coefficients of powers of $\eta$ must be equal.  Taking
coefficient of power $0$ we see that equation \ref{inteq}
is satisfied with all $v_i, u_i, w_i \in E$.
\end{proof}

\begin{lemma}\label{irr-gamma-log}
Let $E$ be a differential field of characteristic $0$.
Assume $F$ is algebraic over $E(\eta)$, $\eta$ is logarithmic
or nonelementary primitive
over $E$, $E$ is algebraically closed in $F$, $F$ and $E$
have the same constants and $E$ is log-explicit.
If $w'u$ is an
irrational gamma term in $F$, then $w'u$ is an
irrational gamma term in $E$.
\end{lemma}

\begin{proof}
$u$ and $w$ are exponential elements in $F$. Since
$\eta$ is a primitive over $E$, by Lemma \ref{ros-struct}
they are in $E$.  $v$ and $(a-1)v - w$ are logarithmic
elements in $F$.  Since $E$ is log-explicit by 
Lemma \ref{ros-struct2} we have
$v = r\eta + \tilde v$ and $(a-1)v - w = q\eta + g$
with rational $r$, $q$ and $\tilde v$, $g$ in $E$.
Since $\eta$ is
transcendental over $F$, this means that $(a-1)r = q$.
But $a$ is irrational, so the last equality is possible
only if $q = r = 0$.  Hence $v\in E$ and $(a-1)v - w \in E$.
So indeed all parts of the term $w'u$ are in $E$.
\end{proof}

\begin{lemma}\label{gamma-str1-prim}
Let $E$ be a log-explicit differential field of characteristic $0$,
$F$ be a differential
field algebraic over $E(\eta)$, $\eta' = u'/u$ with $u\in E$.
Suppose that $F$
has the same constants as $E$ and equation \ref{inteq} holds with $f\in E$,
$v_i, u_i, w_i \in F$, then we can find new
$v_i, u_i, w_i$ algebraic over $E$ and new constant $c_i$
such that \ref{inteq} holds.
\end{lemma}

\begin{proof}
Again, we may assume that $E$  is algebraically closed in $F$.
Since $E$ is log-explicit by Lemma \ref{irr-gamma-log}
terms of sum over $I_4$ are in $E$, so we may disregard them below.
By Lemma \ref{ros-struct} for $i \in I_2 \cup I_3$ we have
$u_i \in E$, $v_i = d_i\eta + {\bar v}_i$ with constant $d_i$ and
${\bar v}_i \in E$.
In particular $v_i \in E(\eta)$ for $i \in I_2 \cup I_3$.
$w_i$ are roots of $v_i$, so taking trace from $F$ to $E(\eta)$
we get $\Tr(w_i) = 0$ if $w_i \notin E(\eta)$ and
$\Tr(w_i) = Mw_i$ where $M=[F : E(\eta)]$ if $w_i \in E(\eta)$.
So taking trace of both sides of equation \ref{inteq} and dividing
by $M$ we get equality with all terms in $E(\eta)$.  Now, $w_i$
are in $E(\eta)$ and $v_i$ are in $E[\eta]$ of degree at most $1$.  Since
$k_i \geq 2$ and $m_i$ is relatively prime to $k_i$ equality
$w_i^{k_i} = v_i^{m_i}$ is only possible when $w_i$ and $v_i$
are in $E$.  Consequently for $i \in I_3$ we have $u_i, v_i, w_i$
in E.  If $i \in I_2$ if $v_i \notin E$, then the corresponding term
is a proper rational function.  Similarly, for $i \in I_1$ by
Lemma \ref{decomp3} we have
$$
\frac{v_i'}{v_i} = \frac{{\bar v}_i'}{{\bar v}_i} + 
\text{proper rational function}
$$
so we can rewrite sum over $I_1$ as sum with terms in $E$ plus
proper rational function.  Writing $v_0$ as sum of polynomial
and proper rational function we see that right hand side of
equation \ref{inteq} can be rewritten as sum with $v_0$ being
a polynomial and other terms in $E$ plus a proper rational
function.  Since decomposition of rational function into
polynomial and proper rational function is unique, this means
that proper rational function above in fact equals $0$.
So now we have $v_0 \in E[\eta]$ and all other terms in $E$.
Standard argument from proof of Liouville theorem shows
that $v_0$ must be a polynomial of degree at most $1$.
Term of degree one in $v_0$ has form $a\eta$ with $a$ being
a constant, and
$$
a\eta' = a\frac{u'}{u}
$$
so derivative of this term can be added to sum over $I_1$.
Consequently, we have all terms in $E$.
\end{proof}

\begin{proof}(Of theorem \ref{gamma-str0}).  First, note that
if \ref{inteq} holds in a gamma extension $F$, than
it holds in a gamma extension having constants algebraic
over constants of $E$.  Namely, $F = E(\theta_1, \dots, \theta_n)$
where each $\theta_i$ satisfies differential equation over
$E(\theta_1, \dots, \theta_{i-1})$ (note that algebraic
equation is treated as differential equation of order $0$).
In other words,
$F = K(\theta_1, \dots, \theta_n)$ is a gamma extension of
$E$ if and only if $\theta_1, \dots, \theta_n$ satisfy
appropriate system of differential equations.  $f = \gamma'$
is also a differential equation.  Like in Lemma 2.1 \cite{SSC}
clearing denominators
we convert system of differential equations to a differential
ideal $I$ plus an inequality $g \ne 0$ (which is responsible
for non-vanishing of denominators) and use result of Kolchin
which says that differential ideal $I$ which has zero in some
extension satisfying $g \ne 0$ has zero in extension having
constants algebraic over constants of $E$ and satisfying
$g \ne 0$.  Solution clearly gives us gamma extension
with constants algebraic over constants of $E$ such that
$f = \gamma'$ has solution.  Next we inductively prove
that for each $i$ equation \ref{inteq} holds
with terms in a field algebraic over
$E_i = E(\theta_1, \dots, \theta_i)$.  For $i = n$ this is
our assumption.  Before inductive step let us remark
that each $E_i$ is log-explicit: this follows by an
easy induction using Lemma \ref{l-closed1} and
Lemma \ref{l-closed2}.
To pass from $i$ to $i - 1$ note
that if $\theta_i$ is algebraic than the claim is
trivial.  When $\theta_i$ is an exponential or a logarithm
we may use Lemma \ref{gamma-str1-exp} (\ref{gamma-str1-prim}
respectively).  If $\theta_i$ is a Ei term or gamma term,
we argue like in case of logarithm, but in the last step
we get extra term in sum over $I_2$ (or $I_3$ or $I_4$).  So
by induction principle equation \ref{inteq} holds
with all terms algebraic over $E$.

Note that assumption that $E$ is log-explicit was used only
in the proof of Lemma \ref{gamma-str1-prim} to handle irrational
gamma terms: if such terms are absent, then we can drop
assumption that $E$ is log-explicit.  Also, if irrational
gamma terms are not present in original integral, then
they will not appear during the proof.  This proves the
clam when $f$ has integral in rational gamma extension.
\end{proof}

\begin{theorem}\label{gamma-str1}
Let $K, E$ be differential fields of characteristic $0$ such that
$K$ is l-closed and log-explicit, has a constant field $C$ and
is algebraically closed in $E$, $E$ is
algebraic over $K(\theta)$ and $\theta$ is an
exponential, a logarithm or a nonelementary primitive monomial.
If $\theta$ is an exponential additionally assume that
for any algebraic extension $L$ of $K$
exponential elements in $L$ together with $\theta$
generate group of exponential elements in $LE$.
If $f$ has
integral in a gamma extension of $E$, then
there is algebraic extension $\bar K$ of $K$
and an algebraic extension $\bar C$ of constants $C$ such that
in Theorem \ref{gamma-str0} we can require that
$a_i, c_i\in \bar C$, 
$v_0, v_i\in {\bar C}E$ for $i \in I_1$,
$v_i \in K(\theta) + {\bar C}$ for
$i \in I_2\cup I_3$, $v_i, w_i \in K + {\bar C}$ for $i \in I_4$,
$u_i\in K(\theta)$ for $i \in I_2 \cup I_4$, $u_i \in {\bar K}(\theta)$,
$w_i\in {\bar K}E$,
$w_iu_i \in {\bar C}E$ for
$i \in I_3$.
If $\theta$ is an exponential or a nonelementary primitive,
then we can take $v_i \in K + {\bar C}$, $i \in I_2\cup I_3$,
$w_i \in {\bar K}$, $w_iu_i \in {\bar C}K(\theta)$ for
$i \in I_3$.
If $\theta$ is a
nonelementary primitive, then we can take $u_i\in K$
for $i \in I_2 \cup I_4$, $u_i \in {\bar K}$, $u_iw_i \in {\bar C}K$
for $i \in I_3$.  If $\theta$ is a logarithm we can take
$u_i\in K$ for $i \in I_4$.

If $f$ has integral in a rational gamma extension of $E$,
then sum over $I_4$ does not appear in the result and we
can drop assumption that $E$ is log-explicit.
\end{theorem}

\begin{proof}
The proof uses method from proof of Theorem 4.1 in \cite{SSC}.
By Theorem \ref{gamma-str0} equation \ref{inteq} holds with
$c_i$, $v_i, u_i, w_i$ algebraic over $E$.
Let $\bar C$
be field of constants of algebraic extension of $E$ generated by
$c_i, v_i, u_i, w_i$.  By assumption $K(\theta)$ has the
same constants as $K$, so $\bar C$ is algebraic over $C$.

Note that by
Lemmas \ref{l-closed1} and \ref{l-closed2} $K(\theta)$ is
l-closed.
For $i \in I_2 \cup I_3 \cup I_4$
corresponding $v_i$ are logarithmic elements in algebraic
extension of $K(\theta)$ so by Lemma \ref{log-closed}
we have $v_i \in K(\theta) + {\bar C}$.  Similarly
if $i \in I_4$ then $(a-1)v_i - w_i \in K(\theta) + {\bar C}$ so
also $w_i\in K(\theta) + {\bar C}$.
For $i \in I_4$ by Lemma \ref{irr-gamma-exp} and \ref{irr-gamma-log}
we have $v_i, w_i$ algebraic over $K$.  Consequently
since $v_i, w_i \in K(\theta) + {\bar C}$ we have $v_i, w_i \in K + {\bar C}$
for $i \in I_4$.
  Also for $i \in I_2 \cup I_3 \cup I_4$
for each $u_i$ some power is in $K(\theta)$ modulo multiplicative
constants.  If needed changing constants $c_i$ in equation \ref{inteq}
we may assume that some power of $u_i$ is in $K(\theta)$,
so $u_i$ are radicals over $K(\theta)$.  By definition $w_i$
for $i \in I_3$
are radicals over $K(\theta) + {\bar C}$.  Note that all terms
of sums over $I_2\cup I_2 \cup I_4$ in equation \ref{inteq} are
in ${\bar C}K(\theta, \{u_i\}, \{w_i\})$, so taking trace
we may assume that all terms are in $E_2 = {\bar C}E(\theta, \{u_i\}, \{w_i\})$.
Next, $E_2$ is an abelian
extension of ${\bar C}E$, so Galois group will act on $u_i$ and
$w_i$ multiplying them by roots of unity.  Consequently
after taking trace from $E_2$ to ${\bar C}E$ terms of the sum
over $i \in I_2 \cup I_4$ will vanish for $i$ such that $u_i \notin {\bar C}E$.
Similarly, terms of the sum over $i \in I_3$ will vanish for $i$
such that $w_iu_i \notin {\bar C}E$.  So we have $v_0 \in {\bar C}E$ and
$v_i\in {\bar C}E$ for $i \in I_1$,
$u_i \in {\bar C}E$ for $i \in I_2$,
$w_iu_i \in {\bar C}E$ for $i \in I_3$.
If $\theta$ is not an exponential, then by Lemma \ref{ros-struct},
$u_i$ are algebraic
over $K$, so taking ${\bar K} = {\bar C}K(\{u_i\})$ where $i \in I_3$ by
definition we have $u_i \in {\bar K}E$ and since $w_i = (w_iu_i)/u_i$
and $w_iu_i \in {\bar C}E$ we have $w_i \in {\bar K}E$.
If $\theta$ is an exponential, than $v_i$ for $i \in I_3$ are algebraic
over $K$, so also $w_i$ are algebraic over $K$ and
taking ${\bar K} = {\bar C}K(\{w_i\})$ where $i \in I_3$ by
definition we have $w_i \in {\bar K}E$ and since $u_i = (w_iu_i)/w_i$
and $w_iu_i \in {\bar C}E$ we have $u_i \in {\bar K}E$.

Since $K(\theta)$ is
l-closed $u_i'/u_i \in K(\theta)$ for $i\in I_2 \cup I_4$.  By Lemma
\ref{exp-const} this means that $u_i$ are in $E$ modulo
multiplicative constants.  So if needed changing $c_i$-s in
equation \ref{inteq} we may take $u_i \in E$ for $i\in I_2 \cup I_4$.
If $\theta$ is not an exponential
this means that $u_i$ are algebraic over $K$, so since $K$
is algebraically closed in $E$, we have $u_i \in K$.
If $\theta$ is an exponential, then by assumption any
exponential element in $E$ is in $K(\theta)$.  So, in
both cases we have $u_i \in K(\theta)$ for $i \in I_2 \cup I_4$.

If $\theta$ is an exponential, then Lemma \ref{ros-struct}
implies that $v_i$ for $i \in I_2\cup I_3$ are algebraic
over $K$, so $v_i \in K + {\bar C}$.  Similarly, if
$\theta$ is a nonelementary primitive, than proof of Lemma
\ref{l-closed1} implies that $v_i$ for $i \in I_2\cup I_3$ are algebraic
over $K$, so again $v_i \in K + {\bar C}$.  Consequently
in both cases $w_i$ for $i \in I_3$ are algebraic over $K$ and we added them to
$\bar K$.

If $\theta$ is an exponential, then by assumption $u_i$ for $i \in I_3$
being exponential elements in ${\bar K}E$ are in ${\bar K}(\theta)$.
If $\theta$ is not an exponential, then $u_i$ for $i \in I_3$
are algebraic over $K$ and we added them to
${\bar K} \subset {\bar K}(\theta)$.  So in both cases $u_i$ for $i \in I_3$
are in ${\bar K}(\theta)$.

If $\theta$ is an exponential or nonelementary primitive, then
$w_i$ are in $\bar K$ and $u_i \in {\bar K}(\theta)$ for $i \in I_3$.
But we also showed that $w_iu_i \in {\bar C}E$.  Consequently
$w_iu_i \in {\bar C}E \cap {\bar K}(\theta) = {\bar C}K(\theta)$
(the last equality follows because ${\bar C}K$ is algebraically
closed in ${\bar C}E$).

If $\theta$ is a nonelementary primitive, then we already
showed that $u_i$ and $w_i$ for $i\in I_3$ are in ${\bar K}$,
so also $w_iu_i \in {\bar K}$ and
$w_iu_i \in {\bar K} \cap {\bar C}E = {\bar C}K$.

Note that when $f$ has integral in a rational gamma extension of $E$,
Theorem \ref{gamma-str0} says that sum over $I_4$ does not appear
in the result.  Our proof does not introduce new terms, so this
remains valid.  Similarly, assumption that $E$ is log-explicit is
only needed to use Theorem \ref{gamma-str0}, so if
$f$ has integral in a rational gamma extension of $E$, then
we can drop it.
\end{proof}

\begin{theorem}\label{gamma-str2}
Let $K, E$ be differential fields of characteristic $0$ such that
$K$ is l-closed and log-explicit and
is algebraically closed in $E$, $E$ is
algebraic over $K(\theta)$ and $\theta$ is an
exponential.  Assume that for any algebraic extension $L$ of $K$
exponential elements in $L$ together with $\theta$
generate group of exponential elements in $LE$.
If $f$ has integral in a gamma extension of $E$, then
there exists $f_1\in E$ and $f_2 \in K[\theta, \theta^{-1}]$
such that $f = f_1 + f_2$, $f_1$ has integral elementary over $E$,
$f_2$ has integral in a gamma extension of $K(\theta)$.  Moreover,
there is effective procedure to find $f_1$ and $f_2$.

If $f$ has integral in a rational gamma extension of $E$,
then $f_2$ has integral in a rational gamma extension of $K(\theta)$
and we can drop assumption that $E$ is log-explicit.
\end{theorem}

\begin{proof}
By Theorem \ref{gamma-str1} there exists algebraic extension
$\bar K$ of $K$ such that equation
\ref{inteq} holds with $v_i \in \bar K$, $u_i \in {\bar K}(\theta)$
for $I_2\cup I_3 \cup I_4$, $w_i \in {\bar K}$.  Since every
exponential element $u$ in ${\bar K}E$ is of form ${\bar u}\theta^{j}$
with integer $j$ and ${\bar u} \in {\bar K}$ we see that
$u_j$ in fact are in ${\bar K}[\theta, \theta^{-1}]$.  Consequently
sum over $I_2\cup I_3 \cup I_4$ is in ${\bar K}[\theta, \theta^{-1}]$.
Denote by $g_2$ sum over $I_2\cup I_3 \cup I_4$ and by $g_1$ sum
of $v_0'$ and sum over $I_1$.  Then $f = g_1 + g_2$, $g_1$ has
elementary integral, $g_1 \in {\bar K}E$, $g_2$ has integral in
a gamma extension of $E$ and $g_2 \in {\bar K}[\theta, \theta^{-1}]$.
Without loss of generality we may assume that ${\bar K}E$ is
a finite algebraic extension of $E$.
Let $m$ be degree of ${\bar K}E$ over $E$ and put
$f_1 = \frac{1}{m}\Tr(g_1)$, $f_2 = \frac{1}{m}\Tr(g_2)$ where $\Tr$
denotes trace from ${\bar K}E$ to $E$.  Of course $f = f_1 + f_2$ and
$f_1$ has integral elementary over $E$.  Note that $f_2$ has integral in
a gamma extension of $E$.  Namely, trace is sum of terms like in
sum over $I_2\cup I_3 \cup I_4$, but we also get conjugates of terms
from original sum.  In particular each term is integrable in a gamma
extension of $E$ so the whole sum is integrable.  Also, gamma
extension of $E$ is a gamma extension of $K(\theta)$, so $f_2$
is integrable in a gamma extension of $K(\theta)$.  We claim that
$f_2 \in K[\theta, \theta^{-1}]$.  Namely, since $\theta \in E$
trace acts separately on each coefficient of $g_2$.  $K$ is
algebraically closed in $E$ so coefficients of $f_2$ have
values in $K$, that is $f_2 \in K[\theta, \theta^{-1}]$.

To compute some decomposition of this sort we may assume
that $E$ is a finite algebraic extension of $K(\theta)$.
Let $n$ be be degree of $E$ over $K(\theta)$.  Put
$h_2 = \frac{1}{n}\Tr(f)$, $h_1 = f - h_1$.  Of course
$f = h_1 + h_2$.  Note that $h_2 = \frac{1}{n}\Tr(f_1) + f_2$,
so $h_1 = f_1 - \frac{1}{n}\Tr(f_1)$.  In particular $h_1$
has elementary integral, $h_2 - f_2$ has elementary integral
and $h_2 \in K(\theta)$.  So we reduced problem to functions
from $K(\theta)$.  Next, write $h_2 = w + p$ where
$w \in K[\theta, \theta^{-1}]$ and $p$ is a proper rational
function with denominator relatively prime to $\theta$.
Note that $w = f_2 + w_1$, where $w_2$ is defined by
equality $h_2 - f_2 = w_2 + p$.  Now, $h_2 - f_2$ has
elementary integral, say
$$
h_2 - f_2 = v_0' + \sum c_i{v_i'}{v_i}
$$
By Lemma \ref{decomp2}
$$
\sum c_i{v_i'}{v_i} = a\eta' + \sum c_i{{\bar v_i}'}{\bar v_i} +
\text{proper rational function}
$$
where $a$ is a constant and ${\bar v_i} \in K$/
We also write $v_0 = r + y$ where
$r \in K[\theta, \theta^{-1}]$ and $y$ is proper rational function with
denominator relatively prime to $\theta$.  Now,
$r' \in K[\theta, \theta^{-1}]$ and derivative of $y$ is
again proper rational function.  So together, proper
rational part $p$ of $h_2 - f_2$ has integral elementary over $K(\theta)$,
namely
$$
p = y' + \sum c_i{v_i'}{v_i} - \sum c_i{{\bar v_i}'}{\bar v_i} - a\eta'.
$$
Consequently $h_1 + p$ has integral elementary over $E$ and
$w$ has integral in a gamma extension of $E$.  So
we can take $f_1 = h_1 + p$ and $h_2 = w$.
Since trace
is effectively computable $h_1$ and $h_2$ are effectively
computable.  Also decomposition $h_2 = w + p$ is effectively
computable, so our choice of $f_1$ and $f_2$ is effectively
computable.
The claim about rational gamma extension is clear.
\end{proof}

\begin{lemma}\label{exp-gen}
Let $K$ be a differential field of characteristic $0$.
Let $\theta$ be an exponential
over $K$, that is $\theta = \eta'\theta$ for some $\eta \in K$,
$\theta$ is transcendental over $K$ and $K(\theta)$ and $K$ have
the same constants.  Let $M$ be an algebraic extension of $K$.
Then $\theta$ generates group of exponential elements of
$MK(\theta)$ modulo $M$.
\end{lemma}
\begin{proof}
Note that $M$ is algebraically closed in $M(\theta)$
and that $M$ and $M(\theta)$ have the same constants.
Let $u$ be an exponential element in $M$.
By Lemma \ref{ros-struct}
there exist integers $n$ and $j$ such that $u^n/\theta^j \in M$.
That is $u^n = s\theta^j$ with $s \in M$.
When $l>0$, this means $u^n \in M[\theta]$.  Since $M[\theta]$ is
a unique factorisation domain we have $u \in M[\theta]$ and
$u = t\theta^m$ for some $t\in M$ and integer $m$.  Similar
argument works for negative $l$, giving the claim.
\end{proof}

\begin{theorem}\label{gamma-str3}
Let $K$ be a differential field of characteristic $0$ which
is l-closed and log-explicit.  Let $\theta$ be an exponential
over $K$, that is $\theta = \eta'\theta$ for some $\eta \in K$,
$\theta$ is transcendental over $K$ and $K(\theta)$ and $K$ have
the same constants.  If $f = \sum a_j\theta^j$ where $a_j \in K$
and $j\in {\mathbb Z}$
have an integral in a gamma extension of $K(\theta)$, then
for each $j$ term $a_j\theta^j$ have integral in a gamma
extension of $K(\theta)$.

If $f$ has integral in a rational gamma extension of $K(\theta)$
then $a_j\theta^j$ have integrals a rational gamma extension of $K(\theta)$
and we can drop assumption that $E$ is log-explicit.
\end{theorem}
\begin{proof}
By Lemma \ref{exp-gen} assumptions of exponential case of
Theorem \ref{gamma-str1} are satisfied.  So there exists
algebraic extension $\bar K$ of $K$ such that equation
\ref{inteq} holds with $v_i \in \bar K$, $u_i \in {\bar K}(\theta)$
for $I_2\cup I_3 \cup I_4$, $w_i \in {\bar K}$.  Since
$f \in K[\theta, \theta^{-1}]$ terms of sum over $I_1$ can
be replaced by a sum with terms in $\bar K$.  Similarly
we can take $v_0 \in \bar K[\theta, \theta^{-1}]$.  Since every
exponential element $u$ in ${\bar K}(\theta)$ is of form ${\bar u}\theta^{j}$
with integer $j$ and ${\bar u} \in {\bar K}$ we see
that we can group terms in equation \ref{inteq} according to
powers of $\theta$ which gives the result.
\end{proof}

It is tempting to weaken assumption about $\theta$ in exponential
case of Theorem \ref{gamma-str1} and in Theorem \ref{gamma-str2}.
However, the following example shows that assumption that
$\theta$ generates group of exponential elements in $E$ modulo
exponential elements in $K$ is too weak.

\begin{example}
This is variant of Example 1.2 from \cite{SSC}.
Let $K={\mathbb Q}(x, \exp(x))$.  Put $f = \exp(-x/2)/\sqrt{x/2}$ and
$E = K(f)$.  $f^2 = 2\exp(-x)/x \in K$, so $E$ is an
algebraic extension of $K$.  Put $g = \sqrt{\pi}\erf(\frac{x}{2}))$.
We have
$$
(2g)' = \frac{\exp(\frac{-x}{2})}{\sqrt{\frac{x}{2}}} = f
$$
so $f$ has integral in a gamma extension of $E$.  It is easy to
check that group of exponential elements in $E$ modulo multiplicative
constants is generated by $\exp(x)$.  However,
trace of $f$ from $E$ to $K(\theta)$ is $0$, so conclusion
of Theorem \ref{gamma-str2} does not hold.
\end{example}

\section{Conclusion}

Our Theorem\ref{gamma-str0} looks like a modest extension of Theorem 1.1
in \cite{SSC} and Theorem 3.1 in \cite{LeLa}, however it covers
case of particular practical
interest.  Theorem \ref{gamma-str1} extends Theorem 4.1
in \cite{SSC}.  We allows larger class of integrals
and it is worth noting that the class of
l-closed fields allowed in Theorem \ref{gamma-str1} is much larger
than class of Liouvillian fields allowed in \cite{SSC}.
We also get more information about needed algebraic extensions,
but the price for this is more complicated statement of the
theorem.  In case when top transcendental is an exponential
(which is most important case) Theorem \ref{gamma-str2}
gives simple statement.  Also Theorems \ref{gamma-str2}
and \ref{gamma-str3}
are a first step towards an algorithm: they effectively reduce
problem of integrating elements of algebraic extension
of $K(\theta)$ where $\theta$ is an exponential over $K$
to integrands of form $a\theta$ where $a$ is in $K$.

\end{document}